\documentclass[12pt]{amsart}

\usepackage{amsmath}
\usepackage{amsthm}
\usepackage{amssymb}
\usepackage{amscd}
\usepackage[pdftex]{graphicx}

\newtheorem{lem}{Lemma}[section]
\newtheorem{thm}[lem]{Theorem}

\theoremstyle{definition}
\newtheorem{defn}[lem]{Definition}

\newtheorem{rem}[lem]{Remark}

\newcommand{\CC}{{\mathbb C}}

\newcommand{\RR}{{\mathbb R}}

\newcommand{\ZZ}{{\mathbb Z}}

\newcommand{\dsps}{\displaystyle}

\newcommand{\Leg}[2]{\left( \frac{#1}{#2} \right)}

\title[Optimal ambiguity functions]
{Optimal ambiguity functions and Weil's exponential sum bound}

\date{}


\author{John J. Benedetto}
\address{Norbert Wiener Center\\
         Department of Mathematics \\
         University of Maryland \\
         College Park, MD 20742 \\
         USA}
\email{jjb@math.umd.edu}
\urladdr{http://www.math.umd.edu/\textasciitilde jjb}

\author{Robert L. Benedetto}
\address{Department of Mathematics \\
         Amherst College \\
         Amherst, MA 01002 \\
         USA}
 \email{rlb@math.amherst.edu}
 \urladdr{http://www.cs.amherst.edu/\textasciitilde rlb}

\author{Joseph T. Woodworth}
\address{Norbert Wiener Center\\
         Department of Mathematics \\
         University of Maryland \\
         College Park, MD 20742 \\
         USA}
\email{joseph.woodworth@gmail.com}

\begin{document}

\newcounter{bean}

\begin{abstract}
Complex-valued periodic sequences,
$u$, constructed by G\"{o}ran Bj\"{o}rck,
are analyzed with regard to the behavior of their discrete periodic
narrow-band ambiguity functions $A_p(u)$.
The Bj\"{o}rck sequences, which are defined on $\ZZ/p\ZZ$ for $p>2$ prime,
are unimodular and have zero autocorrelation on
$(\ZZ/p\ZZ)\smallsetminus\{0\}$. These two properties give rise to the
acronym, CAZAC, to refer to constant amplitude zero autocorrelation
sequences.
The bound proven is
$|A_p(u)| \leq 2/\sqrt{p} + 4/p$ outside of $(0,0)$,
and this is of optimal magnitude given the constraint that $u$
is a CAZAC sequence. The proof requires the full power of Weil's
exponential sum bound, which, in turn, is a consequence of his proof of
the Riemann hypothesis for finite fields.
Such bounds are not only of
mathematical interest, but they have direct applications as sequences in
communications and radar, as well as when the sequences are used as
coefficients of phase-coded waveforms.
\end{abstract}

\maketitle


\section{Introduction}
\label{sec:intro}


\subsection{Purpose}
\label{sec:purp}
Let $\ZZ$ denote the ring of integers and let $\CC$ denote the 
field of complex numbers.
Given an integer $N$, form the ring $\ZZ/N\ZZ$ of integers
modulo~$N$.
\begin{defn} 
\label{def:corr}
Let $u:\ZZ/N\ZZ :\rightarrow\CC$ be an $N$-periodic sequence. The 
\emph{discrete narrow band ambiguity function},
$A_N(u):\ZZ/N\ZZ\times \ZZ/N\ZZ \to \CC$, is defined to be
$$
A_N(u)[m,n]=
\frac{1}{N}\sum_{k=0}^{N-1}u[m+k]\overline{u[k]}e^{-2\pi ikn/N}
$$
for all $(m,n)\in\ZZ/N\ZZ\times\ZZ/N\ZZ$.

The \emph{discrete autocorrelation} of $u$ is the function 
$A_N(u)[\cdot,0]:\ZZ/N\ZZ\rightarrow\CC.$
\end{defn}

The ambiguity function in Definition~\ref{def:corr}
stems from P.~M.~Woodward's definition of the
\emph{narrow band ambiguity function} defined on $\RR\times\RR$ 
\cite{Wood1953a}.

\begin{defn}
\label{def:cazac}
An $N$-periodic sequence $u:\ZZ/N\ZZ\rightarrow\CC$ is
\emph{constant amplitude zero autocorrelation} (\emph{CAZAC}) if
it satisfies the following properties:
\begin{align*}
\text{(CA)} \qquad\qquad &
|u[k]|=1 \quad \text{for all } k\in\ZZ/N\ZZ, \quad \text{and}
\\
\text{(ZAC)} \qquad\qquad &
C(u)[m]=\frac{1}{N}\sum_{k=0}^{N-1}u[m+k]\overline{u[k]}=0
\quad \text{for all } m\in\ZZ/N\ZZ\smallsetminus\{0\}.
\end{align*}
\end{defn}
Clearly, $C(u)[m]=A_N(u)[m,0]$ for each $m\in \ZZ/N\ZZ$. 
Equation~(CA) is the condition that $u$ has constant amplitude $1.$ 
Equation~(ZAC) is the
condition that $u$ has zero autocorrelation.

Our setting is almost exclusively limited to the case that $N=p$ is prime.
As such,  $\ZZ/p\ZZ$ is a field.

We shall use a remarkable construction of \emph{CAZAC} sequences $u_p$
of prime length $p$ to
prove optimal behavior of $A_p(u_p)$. The construction is due to 
G\"{o}ran Bj\"{o}rck \cite{bjor1985}(1985), \cite{bjor1990}(1990).
By \emph{optimal behavior}, we mean that 
if $p$ is an odd prime, then
\begin{equation}
\label{eq:basicineq}
  |A_p(u_p)[m,n]| < \frac{2}{\sqrt{p}} + \frac{4}{p} \quad \text{for all }
(m,n) \in (\ZZ/p\ZZ \times \ZZ/p\ZZ) \smallsetminus \{(0,0)\},
\end{equation}
see Theorem~\ref{thm:Mbound}.
By comparison, 
a short and elementary calculation shows that
for any \emph{CAZAC} $u$, 
$$\max\big\{|A_p(u)[m,n]| : (m,n) \in (\ZZ/p\ZZ \times \ZZ/p\ZZ)
\smallsetminus \{(0,0)\} \big\}
\geq \frac{1}{\sqrt{p-1}},$$
and therefore the bound~\eqref{eq:basicineq} above
is indeed of optimal order of magnitude.

\begin{rem}
\label{rem:weilproof}
The proof of Theorem~\ref{thm:Mbound} requires Andr\'{e} Weil's 
exponential sum bound, \cite{weil},
which is a consequence of his proof of the Riemann Hypothesis for curves over
finite fields, \cite{weil1948b}, announced in the Comptes Rendus in 1940. 
Further, there are no more elementary means to 
prove the inequality (\ref{eq:basicineq}). In fact, in estimating $A_p(u_p)$,
the critical term to estimate is a Kloosterman sum; and, if there were an 
easier way to bound it by $C/\sqrt{p}$, then there would be an easier way to
prove Weil's bound for Kloosterman sums, which is an essential consequence of
\cite{weil} and which has withstood the test of time vis a vis evolutionary
simplification.
\end{rem}

\begin{rem}
\label{rem:bjorckcode}
Notwithstanding the level of mathematics required to prove the inequality
(\ref{eq:basicineq}), as noted in Remark~\ref{rem:weilproof}, we emphasize 
that our coding and implementation
of Bj\"{o}rck's \emph{CAZAC} sequences is truly elementary.
In this regard, see \cite{BenWoo2011},
as well as earlier Bj\"{o}rck experiments and constructions
by one of the authors, e.g., see \cite{BenKonRan2009} and references
therein, cf.\ Remark~\ref{rem:othercazacs}. In the parlance of
waveform design, Theorem~\ref{thm:Mbound} is an ideal discrete
``thumbtack'' narrow band ambiguity function which can be used to
design ideal phase-coded waveforms devoid of 
any substantial time or doppler coupling
in the continuous narrow band ambiguity function plane.
With regard to hardware implementation of these phase-coded waveforms (as
well as others stemming from low correlation sequences), the power,
bandwidth, and hardware requirements will introduce noise. It is
understood that modifications must be made to the formulation of a given
low correlation sequence to permit implementation while controlling this
noise.
\end{rem}

\subsection{Background}
\label{sec:back}




The study of \emph{CAZAC} sequences and of other sequences related to optimal
autocorrelation behavior has origins in several important applications,
one of the most prominent being in the general area of waveform design
associated with radar and communications, see, e.g.,
according to year of publication
\cite{klau1960, KlaPriDarAlb1960, FraZad1962,tury1968,
vakm1969,chu1972, skol1980,popo1992, mow1996, HelKum1998, AusBar1998,
BelMon2002,LevMoz2004, GolGon2005, BenKonRan2009, HolRicSch2010}.
There are hundreds of articles in this area and so this selection may
seem arbitrary, although several of these references contain focused
lists of contributions and specific applications.
Also see Remark~\ref{rem:othercazacs}.

There are also purely mathematical origins for the construction of
\emph{CAZAC} sequences. One such origin is due to Norbert Wiener,
e.g., see new related constructions in
\cite{BenDon2007, BenDat2010}. Another may
be said to have originated in a question by Per Enflo in 1983. This
particular mathematical path has been documented and built upon by
Bahman Saffari \cite{saff2000}. Enflo's question 
is the following for a given odd prime
$p$. 
\textit{Is it true that the Gaussian sequences},
$u:\ZZ/p\ZZ\rightarrow\CC,$
\textit{defined by}
$$
u[k]=\zeta_{p}^{rk^{2}+sk},\qquad k=0,1,...,p-1,
$$
\textit{where $\zeta_p = e^{2\pi i/p}$, $r,s\in\ZZ$ and $p$ does not divide $r$,
are the only unimodular sequences of length p,
with $u[0]=1$, whose Discrete Fourier Transform (DFT) has modulus $1$}?
This is equivalent to asking whether such
sequences $u$ with $u[0]=1$ are the only bi-unimodular sequences
of odd prime length. Enflo was interested in this because of a problem
dealing with exponential sums.

Enflo's question has a positive answer for $p=3$ and $p=5$.
In 1984, by computer search, Bj\"{o}rck discovered counterexamples
to the Enflo question for $p=7$ and $p=11$, see \cite{bjor1985}. 
Later in 1985,
Bj\"{o}rck saw the role of Legendre symbols in his counterexamples,
and this led to his theorems in \cite{bjor1990}.
It also led to a host of mathematical problems, many still unresolved,
about the number of \emph{CAZAC} sequences for a given length,
see, \cite{haag1996, BenDon2007, BenKonRan2009, BenWoo2011}, as well as a 
several valuable oral and email
communications by Saffari \cite{saff2004}. 

\begin{rem}
\label{rem:othercazacs}
It is relevant to mention a striking recent application of
low correlation sequences to radar in terms of 
compressed sensing \cite{HerStr2009}.
In this case, the authors use Alltop
sequences \cite{allt1980} Theorem~2,
cf. \cite{HeaStr2003} Section~2.1.3.
It then becomes natural to
think in terms of frames generated by Bj\"{o}rck sequences for
extending the high-resolution radar/compressed sensing setting of the
authors of \cite{HerStr2009}.

Another approach 
to the problem addressed in
Section~\ref{sec:purp} is found in \cite{GurHadSoc2008},
cf.~\cite{MauSar}.
The authors obtain bounds comparable to those found herein, but their
class of signals, called the oscillator system, is not necessarily
\emph{ZAC} although excellent cross-correlation criteria are obtained,
something we have not pursued.
More important, from the point of view of
application, the characterization and construction of the 
oscillator system
are decidedly representation theoretic. As such an explicit algorithm
associated with the collection of split tori in \emph{Sp} requires a
Bruhat decomposition.

The companion, \cite{BenWoo2011}, of this paper not only exhibits the
simplicity of implementation stressed in Remark~\ref{rem:bjorckcode},
but also reflects the combinatorial and geometrical complexity in the
ambiguity function domain due to the role of the Legendre symbol in
defining Bj\"{o}rck sequences. Some of this complexity is
characterized by intricate Latin and magic square patterns. Further, the
simplicity of implementation gives rise to useful,
efficient bounds off of small neighborhoods
of $(0,0)$ in the ambiguity function domain for compactly supported
waveforms on $\RR$ having $p$ lags whose coefficients are the elements of
a Bj{\"o}rck sequence $u_p$.
Also, it is not difficult to see that, as with the
oscillator system, there is Fourier invariance of Bj{\"o}rck sequences,
most simply calculated in the $p \equiv 1 \pmod{4}$ case, e.g., 
\cite{popo2010}.

\end{rem}


\subsection{Outline}
\label{sec:outline}
We define Bj\"{o}rck sequences in Section~\ref{sec:bjorck}.
Properties of Kloosterman sums are proven in Section~\ref{ssec:kloo}.
These, in turn, are used along with Weil's results and the proper
decomposition formula to express Bj\"{o}rck sequences 
in the way that allows us to prove
Theorem~\ref{thm:Mbound} in Section~\ref{ssec:Mbound}.
Section~\ref{sec:figures} provides figures
and data which motivated and guided us.


\section{Bj{\"o}rck sequences and multiplicative characters}
\label{sec:bjorck}





For each prime number $p$, 
recall that the \emph{Legendre symbol} modulo~$p$ is the function
$\dsps \chi = \Leg{\cdot}{p}:\ZZ/p\ZZ\rightarrow\{+1,0,-1\}$
given by
$$
\chi[k] = 
\Leg{k}{p}=\begin{cases}
+1 & \text{if } k\equiv m^{2} \pmod{p}
\text{ for some } m\in{\ZZ/p\ZZ}^{\times}, 
\\
0 & \text{if } k\equiv 0 \pmod{p}, 
\\
-1 & \text{if } k\not\equiv m^{2} \pmod{p}
\text{ for all } m\in{\ZZ/p\ZZ}.
\end{cases}
$$
The preimage of $+1$ under the Legendre symbol function is the
set $\mathcal{Q}$ of nonzero \emph{quadratic residues} modulo $p$; and
the preimage of $-1$ under the Legendre symbol function is the set
$\mathcal{Q}^{C}$ of \emph{quadratic nonresidues} modulo $p$. 
Among the many properties of the Legendre symbol,
we shall use the fact that it is a
character of the multiplicative group $(\ZZ/p\ZZ)^{\times}$.
This means that $\chi$, when restricted to $(\ZZ/p\ZZ)^{\times}$,
is a group homomorphism into $\CC^{\times}$;
see \cite{HarWri1959}, Chapters~V and~VI.

\begin{defn}
\label{def:bjorck}
The \emph{Bj\"{o}rck sequence} of length $p$, where $p$ is a prime and
$p\equiv 1 \pmod{4}$, is defined by
$$
u[k]=
 \exp\left( i\theta \chi(k) \right)
=\exp\left( i\theta \left( \frac{k}{p} \right) \right),
\qquad \text{ where }\theta=\arccos\left(\frac{1}{1+\sqrt{p}}\right),
$$
for all $k\in\ZZ/p\ZZ$.

The \emph{Bj\"{o}rck sequence} of length $p$, where $p$ is a prime and
$p\equiv 3 \pmod{4}$, is defined by
$$
u_{p}[k]=\begin{cases}
\exp(i\phi) &  \text{if }k\in\mathcal{Q}^{C}\subseteq(\ZZ/p\ZZ)^{\times},
\text{ where }\phi=\arccos\left(\dfrac{1-p}{1+p}\right),
\\
 1 & \text{otherwise},\end{cases}
$$
for all $k\in\ZZ/p\ZZ$.
\end{defn}

In the case $p\equiv 1 \pmod{4}$, Definition~\ref{def:bjorck} is equivalent
to the following definition for the Legendre symbol sequence
$\{0,1,...,-1,...,1\}$ of length $p$.
We replace the first term $0$ by $1$, every term $1$ by
$$
  \eta = \exp\left(i \arccos\frac{\sqrt{p}-1}{p-1}\right)
 = \frac{1}{\sqrt{p}+1} + i \frac{\sqrt{p + 2\sqrt{p}}}{\sqrt{p}+1},
$$
and every term $-1$ by the complex conjugate of $\eta$;
see~\cite{saff2000} for a modest generalization. As proven by
Bj\"{o}rck and differently in \cite{BenWoo2011}, we obtain a
\emph{CAZAC}, and hence bi-unimodular, sequence with three values,
viz., $1$ at $k=0$, and $\eta$ and $\overline\eta$ at
$k\in(\ZZ/p\ZZ)^{\times}$.

In the case $p\equiv 3 \pmod{4}$, Definition~\ref{def:bjorck}   
is equivalent to the following definition for the Legendre
symbol sequence $\{0,1,...,-1\}$ of length $p$. Replace the first
term $0$ by $1$, and replace every $-1$ by
$$
\xi = \exp\left(i \arccos\frac{1-p}{1+p}\right)=
\frac{1-p}{1+p}+i\frac{2\sqrt{p}}{1+p}.
$$
As proven by Bj\"{o}rck and differently in \cite{BenWoo2011}, we obtain a
\emph{CAZAC},
and hence bi-unimodular, sequence with only two values, viz., $1$ and $\xi$.

\section{The main theorem}
\label{sec:theorem}

\subsection{The Legendre symbol and Kloosterman sums}
\label{ssec:kloo}

\begin{defn}
\label{def:kloo}
Let $p$ be a prime.  For any integers $a,b$,
the quantity
$$
   K[a,b;p]
=\sum_{x\in \ZZ/p\ZZ^{\times}}
\exp\big(2\pi i (ax + bx^{-1})/p\big),
$$
where $x^{-1}$ denotes the multiplicative inverse of $x$
in the field $\ZZ/p\ZZ$, is a \emph{Kloosterman sum.}
\end{defn}

Kloosterman sums are always real-valued, as the following
Lemma states.

\begin{lem}
\label{lem:klooreal}
Let $p$ be a prime.  Then
$K[a,b;p]\in\RR$ for all integers $a,b\in\ZZ$.

\end{lem}

\begin{proof}
By the substitution $y=-x$, we have
\[
\overline{K[a,b;p]}
=\sum_{x\in \ZZ/p\ZZ^{\times}}
e^{-2\pi i (ax + bx^{-1})/p}
=\sum_{y\in \ZZ/p\ZZ^{\times}}
e^{2\pi i (ay + by^{-1})/p}
=K[a,b;p].
\qedhere
\]
\end{proof}

The following classical description
of certain Kloosterman sums was first observed by
Hans Sali\'{e} in equation~(52) of \cite{Sal}, using
a formula of Ernst Jacobsthal from a footnote on page~239 of \cite{Jac}.
Jacobsthal's footnote refers the reader to his 1906 Ph.D.\ thesis,
but fortunately the proof of his formula is not 
difficult to derive.

\begin{lem}
\label{lem:jac}
Fix an odd prime $p$ and an integer $a$ not divisible by $p$.
Let $\dsps \chi = \Leg{\cdot}{p}$ denote the Legendre symbol modulo $p$.
\begin{list}{\rm \alph{bean}.}{\usecounter{bean}}
\item
\textup{(Jacobsthal, 1907)}
Let $F:\ZZ/p\ZZ \to \CC$ be any function.  Then
$$\sum_{x\in \ZZ/p\ZZ^{\times}} F[x + a x^{-1}]
= \sum_{x=0}^{p-1} F[x]
+ \sum_{x=0}^{p-1} \chi[x^2 - 4a] F[x].$$
\item
\textup{(Sali\'{e}, 1932)}
$\dsps K[1,a;p] = \sum_{x=0}^{p-1} \chi[x^2 - 4a] e^{2\pi i x/p}$.
\end{list}
\end{lem}

The formulas of Lemma~\ref{lem:jac} are known, but we include
their proofs because of the role they play in our approach.

\begin{proof}
(a). Let $g:(\ZZ/p\ZZ)^{\times} \to \ZZ/p\ZZ$ be the function
$g[x] = x + ax^{-1}$.  For each $t\in \ZZ/p\ZZ$, set
$$N[t] = {\rm card} \{x\in (\ZZ/p\ZZ)^{\times} : g[x] = t\}.$$
The desired sum can now be written as
$$\sum_{x\in \ZZ/p\ZZ^{\times}} F[x + a x^{-1}]
= \sum_{x\in \ZZ/p\ZZ^{\times}} F\big[g[x]\big]
= \sum_{t\in \ZZ/p\ZZ} N[t] F[t].$$
Thus, it suffices to show that $N[t] = 1 + \chi[t^2 -4a]$.

Note that $g[x] = g[ax^{-1}]$.
Conversely, for any $x,y\in (\ZZ/p\ZZ)^{\times}$
with $g[x]=g[y]$, we must have either $y=x$ or $y=ax^{-1}$,
since $0=g[x]-g[y] = (x-y)(1-ax^{-1}y^{-1})$.
Thus, $N[t]\leq 2$ for all $t\in \ZZ/p\ZZ$,
and $N[t]=1$ if and only if $t=g[x]$ for a point
$x\in (\ZZ/p\ZZ)^{\times}$ such that $x=ax^{-1}$.
This latter condition occurs if and only if 
$x^2=a$ in $\ZZ/p\ZZ$; in that case,
$t=g[x] = g[ax^{-1}] = 2x$, or equivalently, $t^2=4a$.
Thus, if we set
$$S = \{g[x] : x\in (\ZZ/p\ZZ)^{\times} , x^2 \neq a \},$$
then
$$N[t] = \begin{cases}
2 & \text{ if } t\in S,
\\
1 & \text{ if } t^2 = 4a,
\\
0 & \text{ otherwise}.
\end{cases}
$$
Note, on the other hand, that
$$1+\chi[t^2 - 4a] = \begin{cases}
2 & \text{ if } t^2 -4a \text{ is a square in } (\ZZ/p\ZZ)^{\times},
\\
1 & \text{ if } t^2 - 4a = 0 \text{ in } \ZZ/p\ZZ,
\\
0 & \text{ otherwise}.
\end{cases}
$$
Thus, it suffices to show that
\begin{equation}
\label{eq:Snew}
S = \{ t\in\ZZ/p\ZZ : t^2 - 4a
\text{ is a square in } (\ZZ/p\ZZ)^{\times} \}.
\end{equation}

Given $t\in S$, pick $x\in (\ZZ/p\ZZ)^{\times}$ such that $t=g[x]$.
Then
$$t^2 - 4a = (x+ax^{-1})^2 - 4a
= x^2 - 2a + a^2 x^{-2} = (x-ax^{-1})^2.$$
In addition, since $t^2\neq 4a$ for all $t\in S$, it follows
that $t^2 - 4a$ is a square in $(\ZZ/p\ZZ)^{\times}$,
proving the forward inclusion.

Conversely, given $t\in (\ZZ/p\ZZ)$ for which there is
some $z\in (\ZZ/p\ZZ)^{\times}$ with
$z^2 = t^2 - 4a$, set $x=(t+z)/2\in (\ZZ/p\ZZ)^{\times}$.
Then $x(x-z) = (t^2 - z^2)/4 = a$, and therefore
$g(x) = x + ax^{-1} = 2x - z = t$.  It follows that $t\in S$,
proving equation~(\ref{eq:Snew}) and hence part~(a).

Part~(b) is immediate by setting
$F[x] = e^{2\pi i x/p}$ and noting that
$\sum_{x=0}^{p-1} e^{2\pi i x/p} = 0$.
\end{proof}

\begin{thm}
\label{thm:charbound}
Fix an odd prime $p$.
Let $\dsps \chi = \Leg{\cdot}{p}$ be the Legendre symbol modulo $p$.
Then
$$e^{-\pi i mn/p} A_p(\chi)[m,n]\in \RR,
\qquad\text{and}\qquad
|A_p(\chi)[m,n]| \leq \frac{2}{\sqrt{p}},$$
for all $m,n \in \ZZ/p\ZZ \smallsetminus \{0\}$.
\end{thm}

\begin{proof}
Fix $m,n \in \ZZ/p\ZZ \smallsetminus \{ 0\}$.
Noting that $\chi$ is multiplicative and real-valued, we have
$$
A_p(\chi)[m,n] = \frac{1}{p}\sum_{k\in\ZZ/p\ZZ}
\chi[k+m]\overline{\chi[k]} e^{-2\pi i kn /p}
= \frac{1}{p} \sum_{k\in\ZZ/p\ZZ} \chi\big[k(k+m)\big] e^{-2\pi i kn /p}.
$$
Let $a=(mn)^2/16$, $b = m/2$, and $c=-1/n$,
where we are doing the arithmetic in $\ZZ/p\ZZ$.
Substituting $k=cx-b$, we have
\begin{align}
\label{eq:Abound}
A_p(\chi)[m,n] & =\frac{1}{p}
\sum_{x\in\ZZ/p\ZZ} \chi\big[ (cx-b)(cx+b) \big] \exp(-2\pi i n (cx-b)/p)
\notag \\
& = \frac{e^{2\pi i bn/p}}{p}
\sum_{x\in\ZZ/p\ZZ} \chi[c^2 x^2 - b^2] e^{2\pi i x/p}
= \frac{e^{2\pi i bn/p}}{p} K[1,a;p],
\end{align}
where the final equality is valid because $b^2 = 4ac^2$, and hence
$$\chi[c^2 x^2 - b^2] = \chi\big[c^2(x^2 - 4a)\big]
= \chi[c^2] \chi[x^2 - 4a] = \chi[x^2 -4a].$$
Since $(e^{2\pi i bn/p})^2 = e^{2\pi i mn/p} = (e^{\pi i mn/p})^2$,
we have
$e^{2\pi i bn/p} = \pm e^{\pi i mn/p}$,
and therefore by equation~\eqref{eq:Abound}
and Lemma~\ref{lem:klooreal},
$$
e^{-\pi i mn/p} A_p(\chi)[m,n] =\pm \frac{1}{p} K[1,a;p] \in\RR.
$$
Finally, because $a\in\ZZ/p\ZZ\smallsetminus\{0\}$,
we have $|K[1,a;p]|\leq 2\sqrt{p}$,
by Weil's bound for Kloosterman sums in \cite{weil}.
Thus, equation~\eqref{eq:Abound} gives
$|A_p(\chi)[m,n]| \leq 2/\sqrt{p},$ as desired.
\end{proof}

\begin{rem}
\label{rem:weilkloo}
In \cite{weil}, Weil proves his bound for $|K[a,b;p]|$
by first using Lemma~\ref{lem:jac} to rewrite $K[a,b;p]$
as $\sum \chi[x^2 - 4a] e^{2\pi i x/p}$ and then bounding
the new sum.  Philosophically, then, it would be more
direct not to convert the sum to the form $\sum \exp (2\pi i (x + ax^{-1})/p)$.
Nevertheless, we have applied the transformation
in Lemma~\ref{lem:jac} because the latter
form of Kloosterman sums is better known than are the details of
Weil's proof.
\end{rem}

\subsection{Main bound}
\label{ssec:Mbound}

We shall need the following technical lemma, which gives bounds
for the ambiguity function of any sequence that is a function
of the Legendre symbol.

\begin{lem}
\label{lem:Ubound}
Fix an odd prime $p$
and complex numbers $r,s,t\in\CC$.
Let $\chi:\ZZ/p\ZZ\to \CC$ be the Legendre symbol modulo $p$,
and let $U:\ZZ/p\ZZ\to \CC$ be the function
$$U[k] = \begin{cases}
r & \dsps \text{ if } \chi(k)=1,
\\
s & \dsps \text{ if } \chi(k)=-1,
\\
t & \text{ if } k=0.
\end{cases}$$
Set $R=(r+s)/2$, $S=(r-s)/2$,
$T=t-R$, and $\zeta_p = e^{2\pi i /p}$.
Then
$$A_p(U)[m,n] = |S|^2 A_p(\chi)[m,n] +
\frac{1}{p}\big(E_1[m,n] + E_2[m,n]\big)$$
for all $m,n \in \ZZ/p\ZZ \smallsetminus \{0\}$,
where 
$E_1[m,n] = R\bar{T} + \bar{R}T \zeta_p^{mn}$,
and
$$E_2[m,n] = 
\begin{cases}
(S\bar{T} + \bar{S} T \zeta_p^{mn})\chi[m]
+ (R\bar{S} + \bar{R} S \zeta_p^{mn})\chi[n]\sqrt{p}
& \text{ if } p\equiv 1 \pmod{4},
\\
(S\bar{T} - \bar{S} T \zeta_p^{mn})\chi[m]
- (R\bar{S} + \bar{R} S \zeta_p^{mn})i\chi[n]\sqrt{p}
& \text{ if } p\equiv 3 \pmod{4}.
\end{cases}
$$
\end{lem}

\begin{proof}
For any two functions $F,G:\ZZ/p\ZZ \to \CC$, write
$$
   B_p(F,G)[m,n] = \frac{1}{p}
   \sum_{k\in\ZZ/p\ZZ} F[k+m] \overline{G[k]} e^{-2\pi i kn /p}.
$$
Define functions $\eta,\delta:\ZZ/p\ZZ \to \CC$ by
$$
\eta[k] = 1
\qquad
\text{and}
\qquad
\delta[k] = \begin{cases}
0 & \text{ if } k\neq 0,
\\
1 & \text{ if } k = 0.
\end{cases}
$$
Thus,
$ U = R\eta + S \chi + T \delta$,
and hence
\begin{multline*}
B_p(U,U) =
|R|^2 B_p(\eta,\eta)
+ |S|^2 B_p(\chi,\chi)
+ |T|^2 B_p(\delta,\delta)
\\
{} 
+ R\bar{T} B_p(\eta,\delta)
+ \bar{R}T B_p(\delta,\eta)
+ S\bar{T} B_p(\chi,\delta)
+ \bar{S}T B_p(\delta,\chi)
+ R\bar{S} B_p(\eta,\chi)
+ \bar{R}S B_p(\chi,\eta).
\end{multline*}

To compute $B_p(U,U)$, we shall compute
each of these nine terms separately.
Since $m\neq 0$, we have $B_p(\delta,\delta)=0$.
In addition,
$B_p(\eta,\eta)=0$,
since $\sum_{k\in\ZZ/p\ZZ} e^{-2\pi i kn/p}=0$
and $n\neq 0$.
We also have $B_p(\chi,\chi)=A_p(\chi)$ by definition.
Meanwhile, it is immediate that
\begin{align*}
pB_p(\eta,\delta)[m,n] & = 1,
&
pB_p(\delta,\eta)[m,n] & = \zeta_p^{mn},
\\
pB_p(\chi,\delta)[m,n] & = \chi[m],  \quad \text{and}
&
pB_p(\delta,\chi)[m,n] & = \zeta_p^{mn}\chi[-m].
\end{align*}

Next, $pB_p(\eta,\chi)[m,n] = \tau[-n;p]$, where $\tau[a;p]$
is the Gauss sum
$$
\tau[a;p] = \sum_{k\in\ZZ/p\ZZ} \chi[k] e^{2\pi i ak /p}.
$$
However, Gauss proved that $\tau[a;p]=\varepsilon\chi[a]\sqrt{p}$,
where $\varepsilon=1$ if $p\equiv 1 \pmod{4}$,
and $\varepsilon=i$ if $p\equiv 3 \pmod{4}$;
see, for example, Proposition~6.3.1 and Theorem~6.4.1 of \cite{IR}.
Hence, $pB_p(\eta,\chi)[m,n] = \varepsilon\chi[-n]\sqrt{p}$.
Similarly,
\begin{align*}
pB_p(\chi,\eta)[m,n] &= \sum_{k\in \ZZ/p\ZZ} \chi[k+m] e^{-2\pi ikn/p}
= \sum_{j\in \ZZ/p\ZZ} \chi[j] e^{-2\pi i(j-m)n/p}
\\
& = \zeta_p^{mn} \tau[-n;p] = \varepsilon\zeta_p^{mn}\chi[-n]\sqrt{p}.
\end{align*}

Combining the nine computations above,
and noting that
$$\chi[-k]=\chi[-1]\chi[k]\begin{cases}
\chi[k] & \text{ if } p \equiv 1 \pmod{4},
\\
-\chi[k] & \text{ if } p \equiv 3 \pmod{4},
\end{cases}
$$
we have
$B_p(U,U) = |S|^2 A_p(\chi) + (E_1 + E_2)/p$,
where $E_1$ and $E_2$ are the quantities
in the statement of Lemma~\ref{lem:Ubound}.
\end{proof}

The following elementary bound will be needed to prove
the $p\equiv 3 \pmod{4}$ case of Theorem~\ref{thm:Mbound}.

\begin{lem}
\label{lem:realbound}
Let $X,Y\in\RR$, and let $z\in\CC$ with $|z|=1$.  Then
$$|zX + (1-z^2)Y| \leq \sqrt{X^2 + 4Y^2}.$$
\end{lem}

\begin{proof}
Noting that $z\overline{z}=1$, we have
\begin{align*}
\big| z X + (1-z^2) Y \big|
&=
\sqrt{\big(z X + (1-z^2) Y \big)
\big(\overline{z} X + (1-\overline{z}^2) Y \big)}
\\
&=
\sqrt{X^2 + \big(z(1-\overline{z}^2) + \overline{z}(1-z^2)\big)XY
+ (1-z^2)(1-\overline{z}^2) Y^2}
\\
&=
\sqrt{X^2 + |1-z^2|^2 Y^2}
\leq \sqrt{X^2 + 4 Y^2},
\end{align*}
since
$z(1-\overline{z}^2) + \overline{z}(1-z^2)
= z - \overline{z} + \overline{z} - z = 0$
and $|1-z^2|\leq 2$.
\end{proof}

We are now ready to state and prove our main result.

\begin{thm}
\label{thm:Mbound}
Let $p$ be an odd prime, and let $u_p$ be the Bj\"{o}rck
function for $p$.  Then the ambiguity function, $A_p(u_p)$,
defined on $\ZZ/p\ZZ\times\ZZ/p\ZZ$ as
$$
     A_p(u_p)[m,n] =
     \frac{1}{p} \sum_{k\in\ZZ/p\ZZ}
     u_p[k+m] \overline{u_p[k]} e^{-2\pi i kn/p},
$$
satisfies the estimate
$$
    |A(u_p)[m,n]| < \frac{2}{\sqrt{p}} + 
\begin{cases}
\dfrac{4}{p} & \text{ if } p\equiv 1 \pmod{4},
\\
\dfrac{4}{p^{3/2}} & \text{ if } p\equiv 3 \pmod{4},
\end{cases}
$$
for all $(m,n) \in (\ZZ/p\ZZ\times\ZZ/p\ZZ) \smallsetminus\{(0,0)\}$.
\end{thm}

\begin{proof}
Fix $(m,n) \in (\ZZ/p\ZZ\times\ZZ/p\ZZ) \smallsetminus\{(0,0)\}$.
If $m=0$, then $n\neq 0$, and we have
$$A_p(u_p)[0,n] = \frac{1}{p}\sum_{k\in\ZZ/p\ZZ}
u_p[k] \overline{u_p[k]} e^{-2\pi ik n/p}
= \frac{1}{p}\sum_{k\in\ZZ/p\ZZ} e^{-2\pi ik n/p} = 0,$$
since $|u_p[k]|=1$ for all $k\in\ZZ/p\ZZ$.  On the other hand,
if $n=0$, then $m\neq 0$, and we have
$$A_p(u_p)[m,0] = \frac{1}{p}\sum_{k\in\ZZ/p\ZZ}
u_p[k+m] \overline{u_p[k]} = 0,$$
because $u_p$ has zero autocorrelation.
Thus, by the fact that $u_p$ is a \emph{CAZAC}, we may assume
for the remainder of the proof that $m,n\neq 0$.

If $p\equiv 1 \pmod{4}$, then in the notation of
Lemma~\ref{lem:Ubound}, we have
$r=(1+\sqrt{p})^{-1}(1 + i \sqrt{2\sqrt{p} + p})$,
$s=(1+\sqrt{p})^{-1}(1 - i \sqrt{2\sqrt{p} + p})$,
and $t=1$.  Thus,
$$R=\frac{r+s}{2} = \frac{1}{1+\sqrt{p}},
\quad
S=\frac{r-s}{2} = \frac{i\sqrt{2\sqrt{p} + p}}{1+\sqrt{p}}
\quad\text{and}\quad
T = t-R = \frac{\sqrt{p}}{1+\sqrt{p}}.$$
The quantities $E_1$ and $E_2$ in Lemma~\ref{lem:Ubound} are therefore
$$
E_1[m,n] = \frac{\sqrt{p}(1+\zeta_p^{mn})}{(1+\sqrt{p})^2}
$$
and
\begin{align*}
E_2[m,n] &=\frac{1}{(1+\sqrt{p})^2}\Big[
(1-\zeta_p^{mn})\sqrt{p} \sqrt{2\sqrt{p} + p} \cdot i \chi[m]
+ (\zeta_p^{mn}-1)\sqrt{p} \sqrt{2\sqrt{p} + p} \cdot i \chi[n]
\Big]
\\
& = \frac{\sqrt{p}}{(1+\sqrt{p})^2}\Big[
i(1-\zeta_p^{mn})\big(\chi[m]-\chi[n]\big)\sqrt{2\sqrt{p} + p}\Big].
\end{align*}
Noting that $|1+\zeta_p^{mn}|$, $|1-\zeta_p^{mn}|$,
and $|\chi[m]-\chi[n]|$ are each less than or equal to $2$
and that $\sqrt{2\sqrt{p} + p} < \sqrt{1 + 2\sqrt{p} + p} = 1 + \sqrt{p}$,
we obtain
$$|E_1[m,n] + E_2[m,n]| <
\frac{2\sqrt{p}}{(1+\sqrt{p})^2} + \frac{4\sqrt{p}}{1+\sqrt{p}}
< \frac{2\sqrt{p}}{(1+\sqrt{p})^2} + 4.$$
Hence, by Lemma~\ref{lem:Ubound}
and Theorem~\ref{thm:charbound}, we have
\begin{align*}
|A_p(u_p)[m,n]| &\leq 
\frac{2}{\sqrt{p}}|S|^2 
+ \frac{2}{\sqrt{p} (1+\sqrt{p})^2} + \frac{4}{p}
=\frac{2}{\sqrt{p} (1+\sqrt{p})^2}
\big( 2\sqrt{p} + p + 1 \big) + \frac{4}{p}
\\
& = \frac{2}{\sqrt{p} (1+\sqrt{p})^2} (1+\sqrt{p})^2 + \frac{4}{p}
 = \frac{2}{\sqrt{p}} + \frac{4}{p}.
\end{align*}

Similarly, if $p\equiv 3 \pmod{4}$, then
$r=1$, $s=(1+p)^{-1}(1-p + 2 i \sqrt{p})$, and $t=1$, and therefore
$$R=\frac{r+s}{2} = \frac{1}{1-i\sqrt{p}},
\quad
S=\frac{r-s}{2} = \frac{-i\sqrt{p}}{1-i\sqrt{p}},
\quad\text{and}\quad
T=t-R = \frac{-i\sqrt{p}}{1-i\sqrt{p}}.$$
Thus, the quantities $E_1$ and $E_2$ in Lemma~\ref{lem:Ubound} are
$$
E_1[m,n] = \frac{i\sqrt{p}(1-\zeta_p^{mn})}{p+1}
$$
and
$$
E_2[m,n] = \frac{1}{p+1} \Big[
(p - p\zeta_p^{mn}) \chi[m]
- (i\sqrt{p} - \zeta_p^{mn} i\sqrt{p} ) i \chi[n] \sqrt{p}  \Big]
= \frac{p(1-\zeta_p^{mn})}{p+1}\Big[\chi[m] + \chi[n] \Big].
$$
Since $|S|^2 = p/(p+1)$, we have
$$\Big| |S|^2 A_p(\chi)[m,n] + \frac{1}{p} E_2[m,n] \Big|
=\frac{1}{p+1}
\Big|pA_p(\chi)[m,n] + (1-\zeta_p^{mn})(\chi[m] + \chi[n])\Big|.$$
Setting $z=e^{\pi i mn/p}$, $X=e^{-\pi i mn/p}pA_p(\chi)[m,n]$,
and $Y=\chi[m] + \chi[n]$,
so that $X\in\RR$ with $|X|\leq 2\sqrt{p}$
by Theorem~\ref{thm:charbound},
$Y\in \RR$ with $|Y|\leq 2$, 
and $|z|=1$,
Lemma~\ref{lem:realbound} tells us that
$$\Big| |S|^2 A_p(\chi)[m,n] + \frac{1}{p} E_2[m,n] \Big|
\leq \frac{\sqrt{X^2 + 4Y^2}}{p+1}
\leq \frac{\sqrt{4p + 16}}{p+1} = \frac{2\sqrt{p+4}}{p+1}.$$
Hence, by Lemma~\ref{lem:Ubound} and the fact
that $|1-\zeta_p^{mn}|\leq 2$, we obtain
\begin{align*}
|A_p(u_p)[m,n]|
& \leq
\Big| |S|^2 A_p(\chi)[m,n] + \frac{1}{p} E_2[m,n] \Big|
+ \Big| \frac{1}{p} E_1[m,n]\Big|
\\
& \leq \frac{2\sqrt{p+4}}{p+1} + \frac{2}{\sqrt{p}(p+1)}
= \frac{2}{\sqrt{p}(p+1)} \big( \sqrt{p^2 + 4p} + 1 \big)
\\
& \leq \frac{2(p+3)}{\sqrt{p}(p+1)} 
= \frac{2}{\sqrt{p}} + \frac{4}{\sqrt{p}(p+1)}
\leq \frac{2}{\sqrt{p}} + \frac{4}{p^{3/2}} .
\qedhere
\end{align*}
\end{proof}

\begin{rem}
The bounds in Theorem~\ref{thm:Mbound} may be improved very slightly
but at the great expense of simplicity.  For example, if $p\equiv 1 \pmod{4}$,
then the bounds $|1-\zeta_p^{mn}|\leq 2$
and $|1+\zeta_p^{mn}|\leq 2$
could be improved, as obviously
these quantities cannot both be simultaneously close to $2$.
However, the resulting bound is far more complicated to write, and
the savings is only about $2p^{-3/2}$, as illustrated
by considering $\zeta_p^{mn}$ very close to $-1$.  Similarly, removing
the simplification $4\sqrt{p}/(1+\sqrt{p}) < 4$ would also only
save us about $4p^{-3/2}$.
\end{rem}


\section{Figures and table}
\label{sec:figures}


Natural algebraic and analytic calculations convinced us that
the proof of Theorem~\ref{thm:Mbound}
depended on substantial number theoretic results. In parallel,
Figure~\ref{fig:curve}
supported the truth of Theorem~\ref{thm:Mbound}
before we proved it.
\begin{figure}
\includegraphics[width=100mm]{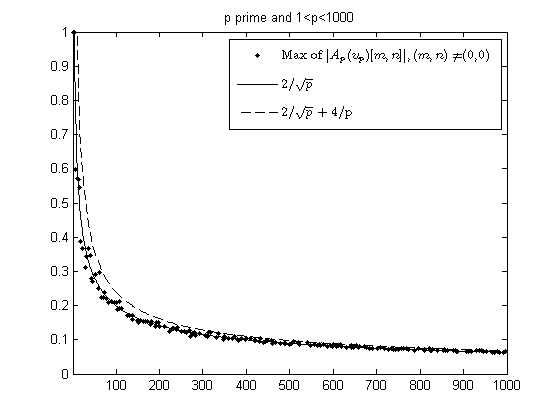}
\caption{$p$ and $\max\{|A_p(u_p)[m,n]|: (m,n) \neq (0,0) \}$}
\label{fig:curve}
\end{figure}
The $x$-axis lists the primes between $1$ and
$1000$. The $y$ axis lists the values,
\begin{equation}
\label{eq:Amax}
     \max_{ (m,n) \neq (0,0)}|A_p(u_p)[m,n]|.
\end{equation}
Figure~\ref{fig:curve} also displays the curves $y=2/\sqrt{p}$
and $y=2/\sqrt{p} + 4/p$ for comparison.
Figure~\ref{fig:doppler},
for the case $p=13$, illustrates the symmetries inherent in the 
function $A_p(u_p)$ on $\ZZ/p\ZZ\times\ZZ/p\ZZ$.
\begin{figure}
\includegraphics[width=60mm]{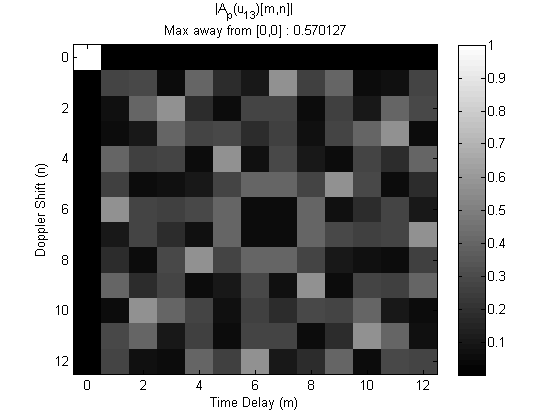}
\caption{$A_p(u_p)$ for $p=13$}
\label{fig:doppler}
\end{figure}
These are fully
explained for all $p$ in \cite{BenWoo2011};
and they led to the realization of the complexity involved in proving 
Theorem~\ref{thm:Mbound},
as well as to a host of geometrical and combinatorial phenomena and
problems. Figure~\ref{fig:3d}
illustrates Theorem~\ref{thm:Mbound}
for the case $p=503$.
\begin{figure}
\includegraphics[width=100mm]{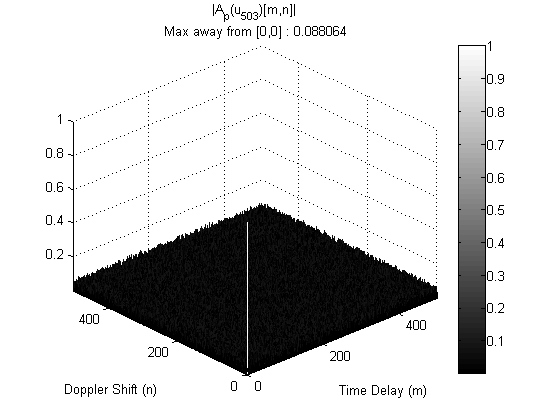}
\caption{$A_p(u_p)$ for $p=503$}
\label{fig:3d}
\end{figure}

Table~\ref{tab:pAp}
indicates some of the finer behavior
of the quantity~\eqref{eq:Amax},
over three different ranges of primes.
This data suggested to us that 
$2/{\sqrt p}$ was very nearly the upper bound for
$|A_p(u_p)[m,n]|$, $(m,n)\neq (0,0)$,
and it helped lead us to the proof that
$2/\sqrt{p} + 4/p$ is an upper bound.
In addition, although a number of primes $p\equiv 1 \pmod{4}$
require a bound larger than $2/\sqrt{p}$, we noted
that only very few primes $p\equiv 3 \pmod{4}$ allowed
$|A_p(u_p)[m,n]|>2/\sqrt{p}$ for $(m,n)\neq (0,0)$.
For example, $p=139$ is the only such prime in Table~\ref{tab:pAp}.
Our broader calculations for other primes
showed that the only such primes 
between $1000$ and $5000$ are $1259$, $2111$, and $3511$;
the only ones between
$10000$ and $24360$ are  $13879$, $16091$ and $23719$;
and there are none between $100000$ and $105000$.
Moreover, for all seven of those primes,
the maximum value of
$|A_p(u_p)[m,n]|-2/\sqrt{p}$ for $(m,n)\neq (0,0)$
is still far smaller than $4/p$, a fact which ultimately
led us to the sharper bound for  $p\equiv 3 \pmod{4}$
in Theorem~\ref{thm:Mbound}.

\begin{table}
\begin{tabular}{|c|c|c||c|c|c|}
\hline
$p$ & $\max |A_p(u_p)|$ & $2/\sqrt{p}$
&
$p$ & $\max |A_p(u_p)|$ & $2/\sqrt{p}$
\\
\hline
3  &  1  &                    1.15470
&
1009  &  0.065505 &   0.062963
\\ \hline
5     &  1                 &     0.894427
&
1013  &  0.064300  &  0.062838
\\ \hline
7     &  0.599074  &     0.755929
&
1019  &  0.060996  &  0.062653
\\ \hline
11    &  0.572765  &     0.603023
&
1021  &  0.063567  &  0.062592
\\ \hline
13    &  0.570127   &     0.554700
&
1031  &  0.061432  &    0.062287
\\ \hline
17    &  0.544798  &     0.485071
&
1033  &  0.062460  &  0.062227
\\ \hline
19    &  0.388357  &     0.458831
&
1039  &  0.061420  &  0.062047
\\ \hline
23    &  0.365960  &  0.417029
&
1049  &  0.063469  &  0.061751
\\ \hline
29    &  0.312280  &  0.371391
&
1051  &  0.060041  &  0.061692
\\ \hline
101   &  0.208395  &  0.199007
&
1061  &  0.063533  &  0.061401
\\ \hline
103   &  0.187876  &  0.197066
&
1063  &  0.060180  &  0.061343
\\ \hline
107   &  0.192309  &  0.193347
&
1069  &  0.062845  &  0.061170
\\ \hline
109   &  0.212120  &  0.191565
&
1087  &  0.059183  &  0.060662
\\ \hline
113   &  0.191960  &  0.188144
&
1091  &  0.059923  &  0.060550
\\ \hline
127   &  0.171881  &  0.177471
&
1093  &  0.060828  &  0.060495
\\ \hline
131   &  0.170530  &  0.174741
&
1097  &  0.063115  &  0.060385
\\ \hline
137   &  0.159752  &  0.170872
&
1103  &  0.059840  &  0.060220
\\ \hline
139   &  0.171326  &  0.169638
&
1109  &  0.061014  &  0.060057
\\ \hline
149   &  0.157303  &  0.163846
&
1117  &  0.062083  &  0.059842
\\ \hline
151   &  0.149263  &  0.162758
&
1123  &  0.058489  &  0.059682
\\ \hline
157   &  0.157840  &  0.159617
&
1129  &  0.062178  &  0.059523
\\ \hline
163   &  0.154913  &  0.156652
&
1151  &  0.058290  &  0.058951
\\ \hline
167   &  0.152243  &  0.154765
&
1153  &  0.061266  &  0.058900
\\ \hline
173   &  0.152966  &  0.152057
&
1163  &  0.058550  &  0.058646
\\ \hline
179   &  0.143966  &  0.149487
&
1171 &   0.056711  &  0.058446
\\ \hline
181   &  0.154193  &  0.148659
&
1181 &   0.059624  &  0.058198
\\ \hline
191   &  0.139244  &  0.144715
&
1187 &   0.057459  &  0.058050
\\ \hline
193   &  0.151468  &  0.143963
&
1193 &   0.059935  &  0.057904
\\ \hline
197   &  0.151479  &  0.142494
&
1201  &  0.057850  &  0.057711
\\ \hline
199   &  0.138516  &  0.141776
&
1213  &  0.058716  &  0.057425
\\ \hline
\end{tabular}
\caption{Comparison of $\max |A_p(u_p)|$
outside $(0,0)$ with $2/\sqrt{p}$}
\label{tab:pAp}
\end{table}

\textbf{Acknowledgements:}
The authors gratefully acknowledge the support of various grants. For
the first-named author, the grants are ONR Grant N00014-09-1-0144 and
MURI-ARO Grant W911NF-09-1-0383. For the second named author, the
grant is NSF Grant DMS-0901494. The third-named author was supported
by the Norbert Wiener Center as recipient of the Daniel Sweet
Undergraduate Research Fellowship. Further, at the time of the
first-named author's presentation of their results at SampTA2011 in
Singapore, Professor Bruno Torresani kindly pointed out two references
on which we comment in Section~\ref{sec:back}.
Finally, and although not represented explicitly in this paper, we have
benefitted from expert advice on hardware implementation by
Drs.~Michael Dellomo, Joseph Lawrence, and George Linde.

\bibliographystyle{amsplain}
\bibliography{BBW}

\end{document}